\documentclass[12pt]{amsart}
\usepackage{amsmath,amssymb,amsthm}    
\usepackage{mathabx}    
\usepackage{mathrsfs}    	
\usepackage[backref,colorlinks]{hyperref}
\usepackage[abbrev,msc-links]{amsrefs}  
\usepackage{enumerate}
\usepackage{enumitem}
\usepackage[symbol]{footmisc}
\usepackage[symbol]{footmisc}
\usepackage{pgf,tikz}
\usepackage{float}
\usepackage{mathrsfs}
\usetikzlibrary{arrows}
\usepackage{rotating}
\usetikzlibrary{calc}
\usetikzlibrary{positioning}
\usetikzlibrary{patterns}
\usetikzlibrary{shapes}
\usetikzlibrary{arrows}
\usetikzlibrary{snakes}
\usepackage{caption}
\usepackage{subcaption}

\normalsize                              
\newtheorem{theorem}{Theorem}[section]
\newtheorem{definition}[theorem]{Definition}
\newtheorem{lemma}[theorem]{Lemma}
\newtheorem{claim}[theorem]{Claim}

\newtheorem{remark}[theorem]{Remark}

\newtheorem{conjecture}[theorem]{Conjecture}

\def\dHH#1{\leavevmode\setbox0=\hbox{#1}\dimen0=\wd0\setbox0=\hbox{.}%
	\advance\dimen0 by -\wd0%
	\hbox{#1\raise-0.5ex\hbox to 0pt{\hss.\kern.5\dimen0}}}%

\begin{document}
	\title[Every Steiner triple system contains an almost spanning \MakeLowercase{d}-ary hypertree]{Every Steiner triple system contains an almost spanning \MakeLowercase{d}-ary hypertree}
		
	\author[A. Arman]{Andrii Arman}
	\address{Department of Mathematics, Emory University, Atlanta, GA 30322, USA}
	\email{andrii.arman@emory.edu}
	\thanks{}
	
	\author[V. R\"{o}dl]{Vojt\v{e}ch R\"{o}dl}
	\address{Department of Mathematics, 
		Emory University, Atlanta, GA 30322, USA}
	\email{rodl@mathcs.emory.edu}
	\thanks{The second  author was supported by NSF grant DMS 1764385}

	\author[M. T. Sales]{Marcelo Tadeu Sales}
	\address{Department of Mathematics, 
		Emory University, Atlanta, GA 30322, USA}
	\email{marcelo.tadeu.sales@emory.edu}
	
	
	\begin{center}
		\today
	\end{center}
	\begin{abstract} 
		In this paper we make a partial progress on the following conjecture: for every $\mu>0$ and large enough $n$, every Steiner triple system $S$ on at least $(1+\mu)n$ vertices contains every hypertree $T$ on $n$ vertices. We prove that the conjecture holds if $T$ is a perfect $d$-ary hypertree.  
	\end{abstract}

	\maketitle

	
	\section{Introduction}
	In this paper we study the following conjecture, raised by second author and Bradley Elliot~\cite{ER}.
	
	\begin{conjecture}\label{conj:ER}
		Given $\mu>0$ there is $n_0$, such that for any $n\geq n_0$, any hypertree $T$ on $n$ vertices and any Steiner triple system $S$ on at least $(1+\mu)n$ vertices, $S$ contains $T$ as subhypergraph.
	\end{conjecture}
	Note that any hypertree $T$ can be embedded into any Steiner triple system $S$, provided $|V(T)|\leq \frac{1}{2}\left(|V(S)|+3\right)$. The problem becomes more interesting if the size of the tree is larger. In~\cite{ARS} (see also~\cite{ER} Section 5) Conjecture~\ref{conj:ER} was verified for some special classes of hypertrees. In this paper we verify the conjecture for another class of hypertrees -- perfect $d$-ary hypertrees.
	
	\begin{definition}
	A perfect $d$-ary hypertree $T$ of height $h$ is a hypertree $T$ with $V(T)=\bigcup_{i=0}^h V_i$, such that $|V_i|=(2d)^i$ for all $i\in[0,h]$ and for each $i\in[0,h-1]$ and $v\in V_i$ there are $d$ disjoint triplets $\{v,u_j, w_j\}$ with $u_j,w_j\in V_{i+1}$, $j\in[d]$.
	\end{definition}
	In other words, $T$ is a perfect $d$-ary hypertree if every non-leaf vertex has $2d$ children (or a forward degree $d$). The main result of this paper is the following theorem.
	
	\begin{theorem}\label{thm:main}
		For any real $\mu>0$ there is $n_0$ such that the following holds for all $n\geq n_0$ and any positive integer $d$. If $S$ is a Steiner triple system with at least $(1+\mu)n$ vertices and $T$ is a perfect $d$-ary hypertree on at most $n$ vertices, then $T\subseteq S$.
	\end{theorem}
	
	\begin{remark}

	One can consider the following extension of Theorem~\ref{thm:main}. In a perfect $d$-ary hypertree label children of every vertex with numbers $\{1,\ldots, 2d\}$. Then every leaf can be identified with a sequence $\{a_1, \ldots, a_h\}\in [2d]^{h}$ based on the way that leaf was reached from the root. We say that $T$ is an almost perfect $d$-ary hypertree if $T$ is obtained from a perfect $d$-ary hypertree by removing the smallest $2t$ leafs in lexicographic order (for some integer $t$)

	With essentially same proof as of Theorem~\ref{thm:main}, for a fixed $d$ any sufficiently large Steiner triple system $S$ contains any almost perfect $d$-ary hypertree $T$ with $|V(T)|\leq |V(S)|/(1+\mu)$. For more extensions of Theorem~\ref{thm:main}, see Section~\ref{sec:remarks}.
	\end{remark}
	
	
	\section{Preliminaries}
	For positive integer $k$ let $[k]=\{1,\ldots, k\}$ and for positive integers $k<\ell$ let $[k,\ell]=\{k, k+1, \ldots, \ell\}$. We write $x=y\pm z$ if $x\in[y-z,y+z]$. We write $A=B\sqcup C$ if $A$ is a union of disjoint sets $B$ and $C$. 
	
	A hypertree is a connected, simple (linear) $3$‐uniform hypergraph in which every two vertices are joined by a unique path. A hyperstar $S$ of size $a$ centered at $v$ is a hypertree on vertex set $v, v_1, v_2, \ldots, v_{2a}$ with edge set $E(S)=\{\{v,v_{2i-1}, v_{2i}\} : i\in[a]\}$. A Steiner triple system (STS) is a $3$‐uniform hypergraph in which every pair of vertices is	contained in exactly one edge.
		
	If $H$ is a hypergraph and $v\in V(H)$, then $d_{H}(v)$ (or $d(v)$ when the context is clear) is the degree of a vertex $v$ in $H$. 
	
	For $V(H)=X\sqcup Y$ we denote by $H[X,Y]$ the spanning subhypergraph of $H$ with $$E(H[X,Y])=\{e\in E(H): |e\cap X|=1, |e\cap Y|=2\}.$$ 
		
	The proof of Theorem~\ref{thm:main} relies on the application of an existence of almost perfect matching in an almost regular 3-uniform simple hypergraph. We will use two versions of such results. In the first version the degrees of a small proportion of vertices are allowed to deviate from the average degree. We will use Theorem 4.7.1 from~\cite{AS} (see~\cite{FR} and~\cite{PS} for earlier versions). 
	\begin{theorem}\label{thm:FR}
		For any $\delta>0$ and $k>0$ there exists $\varepsilon$ and $D_0$ such that the following holds. Let $H$ be a 3-uniform simple hypergraph on $N$ vertices and $D\geq D_0$ be such that
		\begin{itemize}
			\item[(i)] for all but at most $\varepsilon N$ vertices $x$ of $H$ the degree of $x$		$$d(x)=(1\pm\varepsilon)D.$$
			\item[(ii)] for all $x\in V(H)$ we have $$d(x)\leq kD.$$ 
		\end{itemize}
		Then $H$ contains a matching on at least $N(1-\delta)$ vertices. 
	\end{theorem}

	A second version is a result by Alon, Kim and Spencer~\cite{AKS}, where under assumption that all degrees are concentrated near the average, a stronger conclusion may be drawn. We use a version of this result as stated in~\cite{KR}\footnote{We refer to Theorem 3 from that paper. There is a typo in the conclusion part of that theorem, where instead of $O(ND^{1/2}\ln^{3/2}D)$ there should be $O(ND^{-1/2}\ln^{3/2}D)$}.
	\begin{theorem}\label{thm:AKS}
		For any $K>0$ there exists $D_0$ such that the following holds. Let $H$ be a 3-uniform simple hypergraph on $N$ vertices and $D\geq D_0$ be such that $deg(x)=D\pm K\sqrt{D \ln D}$ for all $x\in V(H)$. Then $H$ contains a matching on $N-O(ND^{-1/2}\ln^{3/2}D)$ vertices. 
	\end{theorem}
	
	Here the constant in $O()$-notation is depending on $K$ only an is independent of $N$ and $D$. 
	
	In Lemma~\ref{lemma:partition} we consider a random partition of the vertex set of Steiner triples system $S$ and heavily use the following version of Chernoff's bound (this is Corollary 2.3 of Janson, \L uczak, Rucinski~\cite{JLR}).
	\begin{theorem}\label{thm:Chernoff}
		Let $X\sim \text{Bi}(n,p)$ be a binomial random variable with the expectation $\mu$, then for $t\leq \frac{3}{2}\mu$
		$$\mathbb{P}(|X-\mu|>t)\leq 2e^{-t^{2}/(3\mu)}.$$
	In particular for $K\leq \frac{3}{2}\sqrt{\frac{\mu}{\ln\mu}}$ and $t=K\sqrt{\mu \ln \mu}$
	\begin{equation}\label{eq:chernoff1}
		\mathbb{P}(|X-\mu|>K\sqrt{\mu\ln \mu})\leq 2(\mu)^{-K^2/3}.
	\end{equation}
	If $\varepsilon>0$ is fixed and $\mu>\mu(\varepsilon)$, then
	\begin{equation}\label{eq:chernoff2}
		\mathbb{P}(X=(1\pm\varepsilon)\mu)=1-e^{-\Omega(\mu)}.
	\end{equation}
	\end{theorem}
	
	\section{Proof of Theorem~\ref{thm:main}}
	\subsection{Proof Idea}
	Assume that $S$ is an STS on at least $(1+\mu)n$ vertices. We will choose small constants $\varepsilon \ll \mu$. 
	
	Let $T$ be a perfect $d$-ary tree on at most $n$ vertices with levels $V_i$, $i\in[0,h]$ and let $i_0=\max\{i, |V_i|\leq \varepsilon n \}$. Let $T_0$ be a subhypertree of $T$ induced on $\bigcup_{i=0}^{i_0}V_i$. To simplify our notation we set $t=h-i_0$ and for all $i\in[0,t]$ we set $L_i=V_{i_0+i}$. Our goal is to find $L\subset V(S)$ with $|L|=|V(T)|$ such that $S[L]$, the subhypergraph induced by $L$,  contains a spanning copy of $T$. In particular we would find such $L$, level by level, first embedding $T_0$, and then $L_1, \ldots, L_t$. 
	
	To start we consider a partition $\mathcal{P}=\{C_0,\ldots, C_t, R\}$ of $V(S)$ with ``random-like'' properties (see Lemma~\ref{lemma:partition} for the description of $\mathcal{P}$) in the following way:
	\begin{itemize}
		\item The first few levels of $T$, constituting $T_0$ with the subset of leafs $L_0$, will be embedded greedily into $S[C_0]$.
		\item Then Lemma~\ref{lemma:FR} and Lemma~\ref{lemma:AKS} will be used to find star forests in $S[L_0,C_1]$ and $S[C_{i-1},C_{i}]$ for $i\in[2,t]$. The union of these star forests will establish embedding of almost all vertices of $T$ (see Claim~\ref{claim:sizeT1}).
		\item Finally, the reservoir vertices will be used to complete the embedding (see Claim~\ref{claim:reservoir}).
	\end{itemize} 
	
	\subsection{Auxiliary Lemmas}
	We start our proof with the following Lemma that allows us later to verify that $S[L_0, C_1]$ contains almost perfect packing of size at most $d$ hyperstars centered at vertices of $L_0$
	
	\begin{lemma}\label{lemma:FR}
		For any positive real $\delta$, $k>1$ there are $\varepsilon>0$ and $D_0$ such that the following holds for all $D\geq D_0$ and all positive integers $d$. Let $G=(V,E)$ be a 3-uniform simple hypergraph on $N$ vertices such that $V=X\sqcup Y$ and 
		\begin{itemize}
			\item[(i)] for all $e\in E$, $|e\cap X|=1$ and $|e\cap Y|=2$.
			\item[(ii)] for all vertices $v\in X$ we have
			$$d(v)=dD(1\pm\varepsilon),$$
			and for all but at most $\varepsilon N$ vertices $v\in Y$ we have 
			$$d(v)=D(1\pm\varepsilon).$$
			\item[(iii)] $d(v)\leq kD$ for all $v\in Y$. 
		\end{itemize}
		Then $G$ contains a packing of hyperstars of size at most $d$ centered at vertices of $X$ that covers all but at most $\delta N$ vertices.	
	\end{lemma}
	\begin{proof}
		For given $\delta, k$ set $\delta_{2.1}=2\delta/3$ and $k_{2.1}=k$. With these parameters as an input, Theorem~\ref{thm:FR} yields $\varepsilon_{2.1}$ and $D_0$. Set $\varepsilon=\varepsilon_{2.1}/2$ and note that if Theorem~\ref{thm:FR} holds for some $\varepsilon_{2.1}$ and $D_0$ then it also holds for smaller values of $\varepsilon_{2.1}$ and larger values of $D_0$. Therefore we may assume that $\varepsilon$ is sufficiently small with respect to $\delta$ and $k$.

		
		Let $G$ that satisfies conditions (i)--(iii) be given. We start with constructing an auxiliary hypergraph $H$ that is obtained from $G$ by repeating the following splitting procedure for each vertex $v\in X$: split hyperedges incident to $v$ into $d$ disjoint groups, each of size $D(1\pm 2\varepsilon)=D(1\pm \varepsilon_{2.1})$, and then replace $v$ with new vertices $v_1, \ldots, v_d$ and each hyperedge $\{v,u,w\}$ that belongs to group $j$ with a hyperedge $\{v_j, u,w\}$. 
		
		First, we show that $|V(H)|\leq \frac{3}{2}N$. Note that due to conditions (i)-(iii) the number of hyperedges 
		$$|E(G)|\sim |X|dD \sim \frac{1}{2}|Y|D.$$
		Provided $\varepsilon$ is small enough and $N$ is large enough compared to $\delta$ we can guarantee $|X|\leq \frac{N}{2d}$. Finally, $|V(H)|\leq d|X|+|Y|$ by construction of $H$, so 
		$$N< |V(H)|\leq (d-1)|X|+|X|+|Y|\leq (d-1)|X|+N\leq \frac{3}{2}N.$$
		Hypergraph $H$ satisfies assumptions of the Theorem~\ref{thm:FR} with parameters $\delta_{\ref{thm:FR}}=\frac{2}{3}\delta$, $k_{2.1}=k$, $\varepsilon_{2.1}=2\varepsilon$, $D_{2.1}=D$ and $N_{2.1}=|V(H)|$. Indeed, all of the vertices in $H$ still have degrees at most $k_{2.1}D_{2.1}$, and for all but at most $\varepsilon_{2.1} N_{2.1}$ vertices we have $d_{H}(v)=D_{2.1}(1\pm \varepsilon_{2.1})$. Therefore, there is a matching $M$ in $H$ that omits at most $\delta_{\ref{thm:FR}}N_{2.1}\leq \delta N$ vertices. 
		
		Now, matching $M$ in $H$ corresponds to a collection of hyperstars $S_1, \ldots, S_k$ in $G$ with centers at vertices of $X$ and size of each $S_i$ is at most $d$. Indeed, recall that during the construction of $H$ some vertices $v\in X$ were replaced by $d$ vertices $v_1,\ldots, v_d$, hyperedges incident to $v$ were split into $d$ almost equal in size disjoint groups, and then each hyperedge $\{v,u,w\}$ in $j$-th group was replaced with $\{v_j,u,w\}$. Consequently, a matching in $H$ that covers some vertices $v_i$ gives a rise to a hyperstar centered at $v$ of size at most $d$ in $G$. 
		
		Moreover since $M$ in $H$ omits at most $\delta N$ vertices, the union of hyperstars $S_1,\ldots, S_k$ also omits at most $\delta N$ vertices.  
	\end{proof}
	
	The following Lemma allows us later to verify that $S[C_{i}, C_{i+1}]$ contains almost perfect packing of size at most $d$ hyperstars centered at vertices of $C_i$ for all $i\in[t-1]$.

	\begin{lemma}\label{lemma:AKS}
		For any positive real $K$ there is $D_0$ such that the following holds for all $D\geq D_0$, $\Delta=K\sqrt{D\ln D}$ and any positive integer $d$. Let $G=(V,E)$ be a 3-uniform simple hypergraph on $N$ vertices such that $V=X\sqcup Y$ and 
		\begin{itemize}
			\item[(i)] for all $e\in E$, $|e\cap X|=1$ and $|e\cap Y|=2$.
			\item[(ii)] $d(v)=d(D\pm\Delta)$ for all $v\in X$ and $d(v)=D\pm\Delta$  for all $v\in Y$.
		\end{itemize}
		Then $G$ contains a packing of hyperstars of size at most $d$ centered at vertices of $X$ that covers all but at most $O(ND^{-1/2}\ln^{3/2}D)$ vertices.	
	\end{lemma}
	Here the constant in $O()$-notation depends on $K$ only.
	\begin{proof}
		Proof is almost identical to the proof of Lemma~\ref{lemma:FR}. For a given $K$ let $D_0$ be the number guaranteed by Theorem~\ref{thm:AKS} with $2K$ as input.  
		
		Let $G$ that satisfies conditions (i),(ii) be given. We start with constructing an auxiliary hypergraph $H$ that is obtained from $G$ by splitting every vertex $v\in X$ into $d$ new vertices $v_1, \ldots, v_d$ that have degrees $D\pm 2\Delta$.
		
		First, we will show that $|V(H)|=\Theta(N)$. Note that due to conditions (i) and (ii), we have 
		
		$$\frac{|Y|(D\pm\Delta)}{2}= |E(G)|= |X|d\left(D\pm\frac{\Delta}{2}\right).$$
		
		In particular, $$|Y|= |X|d\left(\frac{2D\pm\Delta}{D\pm\Delta}\right).$$
		As $|X|+|Y|=N$ we have that $|Y|=\Theta(N)$ and hence $d|X|=\Theta(N)$. Then $|V(H)|= d|X|+|Y|$ by construction of $H$, so $|V(H)|=\Theta(N)$ as well.
		
		Hypergraph $H$ satisfies assumptions of the Theorem~\ref{thm:AKS} with parameters $2K$ and $D\geq D_0$. Therefore, there is a matching $M$ in $H$ that omits at most $O(ND^{-1/2}\ln^{3/2}D)$ vertices. 
		
		Now, matching $M$ in $H$ corresponds to a collection of hyperstars $S_1, \ldots, S_k$ in $G$ with centers at vertices of $X$. Each hyperstar $S_i$ contains at most $d$ hyperedges and hyperstars $S_1, \ldots, S_{k}$ cover all but at most $O(ND^{-1/2}\ln^{3/2}D)$ vertices of $G$, which finishes the proof.
	\end{proof}
	
	\subsection{Formal Proof}
	We start by defining constants, proving some useful inequalities and proving Lemma~\ref{lemma:partition}. 
	
	Let $S$ be a Steiner triple system on $m\geq(1+\mu)n$ vertices and let $T$ be the largest perfect $d$-ary hypertree with at most $n$ vertices. Our goal is to show that $T\subset S$. 
	
	We make few trivial observations. First, if $d>\sqrt{n}$ and $T$ is perfect $d$-ary hypertree with $|V(T)|\leq n$, then $T$ is just a hyperstar which $S$ clearly contains. Second, if $m> 2n$, then $T$ can be found in $S$ greedily. Finally, if Theorem~\ref{thm:main} holds for some value of $\mu$, then Theorem~\ref{thm:main} holds for larger values of $\mu$. Hence we may assume without loss of generality that $d\leq \sqrt{n}$, $m\leq 2n$ and $\mu\leq \frac{1}{4}$.
	
	{\bf {Constants.}} 		
	We will choose new constant $\varepsilon<\delta <\rho <\mu$ independent of $m,n$:
	\begin{equation}\label{eq:rho+delta}
		\rho=\left(\frac{3\mu-\mu^2}{8(1+\mu)}\right)^2 , \qquad \delta=\frac{(1+\mu)\rho}{20}.
	\end{equation}
	Let $\varepsilon_{3.1}$ be a constant guaranteed by Lemma~\ref{lemma:FR} with $\delta$ and $k=2$ as an input. We choose $\varepsilon$ to be small enough, in particular we want 
	\begin{equation}\label{ineq:epsilon}
		\varepsilon<\min\{\delta^2, \mu^2/16, (\varepsilon_{3.1})^{10}, 1/10^{100}\}.
	\end{equation}

	{\bf Properties of $T$.} Here we define the levels of $T$ and prove some useful inequalities. Recall that $V_i$, $i\in[0,h]$ denoted the levels of $T$. For $i_0=\max\{i, |V_i|\leq \varepsilon n \}$ let $T_0$ be a subhypertree of $T$ induced on $\bigcup_{i=0}^{i_0}V_i$. To simplify our notation we also set $t=h-i_0$ and for all $i\in[0,t]$, $L_i=V_{i_0+i}$ and $\ell_i=|L_i|$. Then we have for $i\in[t]$
	\begin{equation}\label{ineq:elli}
		\ell_i=(2d)^{i}\ell_0, \qquad \varepsilon n \geq \ell_0 > \frac{\varepsilon}{2d}n.
	\end{equation}
	Since $n\geq \ell_t$, (\ref{ineq:elli}) implies $n\geq (2d)^{t-1}\varepsilon n$ and consequently 
	\begin{equation}\label{ineq:t}
		t\leq 1+\frac{\log\frac{1}{\varepsilon}}{\log (2d)}\leq 1+\log\frac{1}{\varepsilon}.
	\end{equation}
	
	Finally, $T_0=\bigcup_{i=0}^{i_0}V_i$, where $|V_{i_0}|=(2d)^{i_0}=\ell_0$, so 
	\begin{equation}\label{ineq:T_0}
		|V(T_0)| =\frac{(2d)^{i_0+1}-1}{2d-1} \leq \frac{(2d)\ell_0}{2d-1}\stackrel{(\ref{ineq:elli})}{\leq} 2\varepsilon n.
	\end{equation}

	{\bf Partition Lemma.}
	
	For a given STS $S$ with $m$ vertices our goal will be to find a partition $\mathcal{P}=\{C_1,\ldots,C_t,R\}$ of $V(S)$ so that $S[C_0]$ contains a copy of $T_0$ (and $L_0$), sets $C_1, \ldots, C_t$ are the ``candidates'' for levels $L_1, \ldots, L_t$ of $T$ and $R$ is a reservoir. Such a partition will be guaranteed by Lemma~\ref{lemma:partition}.
	
	In the proof we will consider a random partition $\mathcal{P}$, where each vertex $v\in V(S)$ ends up in $C_i$ with probability $p_i$ and in $R$ with probability $\gamma$ independently of other vertices. 
	
	To that end set 
	\begin{equation}\label{eq:p_0}
		p_0=4\sqrt{\varepsilon},
	\end{equation}
	then by~(\ref{ineq:T_0}) and (\ref{ineq:epsilon}) 
	\begin{equation}\label{ineq:p_0}
	\frac{p_0^2}{4}\left(m-1\right)\geq |V(T_0)|, \qquad p_0\leq \frac{\mu}{4}.
	\end{equation} 
	Now, for all $i\in[t]$ define 
	\begin{equation}\label{eq:p_i}
	p_i=\frac{\ell_i}{m}\stackrel{(\ref{ineq:elli})}{\geq} \frac{\varepsilon n}{m}\geq \frac{\varepsilon}{2} .
	\end{equation}
	
	Finally let $\gamma=1-\sum_{i=0}^{t}p_i$. Then 
	$$\gamma\geq 1-\frac{\mu}{4}-\sum_{i=1}^{t}p_i=1-\frac{\mu}{4}-\frac{\sum_{i=1}^{t}\ell_i}{m}\geq 1-\frac{\mu}{4}-\frac{|V(T)|}{m}, $$
	and so 
	\begin{equation}\label{ineq:gamma}
		\gamma\geq 1-\frac{\mu}{4}-\frac{n}{m}\geq 1-\frac{\mu}{4}-\frac{1}{1+\mu}= \frac{3\mu-\mu^2}{4(1+\mu)}\stackrel{(\ref{eq:rho+delta})}{=} 2\sqrt{\rho}.
	\end{equation}
	Hence $\gamma \in (0,1)$.
	
	\begin{lemma}\label{lemma:partition}
	 Let $\varepsilon$, $\ell_0, \ldots, \ell_t$, $p_0, \ldots, p_t$, $\gamma$ and $\rho$ be defined as above. Then for some $m_0=m_0(\varepsilon)$ and $K=8$ the following is true for any $m\geq m_0$. If $S$ is a STS on $m$ vertices, then there is a partition $\mathcal{P}=C_0\sqcup C_1\ldots \sqcup C_t \sqcup R$ of $V(S)$ with the following properties:
	 \begin{itemize}
	 	\item[(a)] $|C_i|=\ell_i\pm K\sqrt{\ell_i\ln \ell_i}$ for all $i\in[t]$.
	 	\item[(b)] for all $i\in[t]$ and all $v\in C_{i-1}$
	 	$$d_{S[C_{i-1},C_i]}(v)=d\left(p_i\ell_{i-1}\pm K\sqrt{p_i\ell_{i-1}\ln p_i\ell_{i-1}}\right).$$
	 	\item[(c)]for all $i\in[2,t]$ and all $v\in C_{i}$
	 	$$d_{S[C_{i-1},C_i]}(v)=p_{i}\ell_{i-1}\pm K\sqrt{p_i\ell_{i-1}\ln p_i\ell_{i-1}}.$$
	 	\item[(d)] for all $v\in V(S)$, $d_{S[v\cup R]}(v)\geq \rho m.$
	 	\item[(e)] $|C_0|=p_0m\pm K\sqrt{p_0m\ln p_0m}$ and $S[C_0]$ contains a copy of a hypertree $T_0$ with $L_0$ as its last level. Moreover for all but at most $\varepsilon^{0.1}|C_1|$ vertices $v\in C_1$
	 	$$d_{S[L_0,C_1]}(v)=(1\pm\varepsilon^{0.1})p_1\ell_0,$$
	 	and for all vertices $v\in C_1$
	 	$$d_{S[L_0,C_1]}(v)\leq2 p_1\ell_0.$$ 
		 \end{itemize}	
	\end{lemma}
	\begin{proof}
		Recall that $\sum_{i=0}^{t}p_i+\gamma=1$. Consider a random partition  $\mathcal{P}=\{C_0,\ldots, C_t, R\}$, where vertices $v\in V(S)$ are chosen into partition classes independently so that $\mathbb{P}[v\in C_i]=p_i$ for $i\in[0,t]$ while $\mathbb{P}[v\in R]=\gamma$. For $j\in\{a,b,c,d,e\}$ let $X^{(j)}$ be the event that the corresponding part of Lemma~\ref{lemma:partition} fails. We will prove that $\mathbb{P}[X^{(j)}]=o(1)$ for each $j\in\{a,b,c,d,e\}$.
		
		{\bf Proof of Property (a).}
		For all $i\in[t]$ let $X_{i}^{(a)}$ be the event that $$\left||C_i|-p_im\right|>K\sqrt{p_im\ln p_im}.$$ Then since $|C_i|\sim \text{Bi}(m,p_i)$ and $\mathbb{E}(|C_i|)=p_im\stackrel{(\ref{eq:p_i})}{=}\ell_i\stackrel{(\ref{ineq:elli})}{=}\Omega(m)$, Theorem~\ref{thm:Chernoff} implies that
		$$\mathbb{P}[X_i^{(a)}]\leq 2(\ell_i)^{-K^2/3}=o(m^{-20}).$$
		Since by (\ref{ineq:t}),  $t\leq1+\log\frac{1}{\varepsilon}\ll m$ we infer that 
		$$\mathbb{P}[X^{(a)}]=\mathbb{P}[\bigcup_{i=1}^{t}X^{(a)}_{i}]\leq \sum_{i=1}^{t}\mathbb{P}[X_{i}^{(a)}]=o(1).$$
		
		{\bf Proof of Property (b).}
		For all $i\in[t]$ and $v\in V(G)$ let $X_{i,v}^{(b)}$ be the event
		$$\left|d_{S[C_{i-1},C_{i}]}(v)-dp_{i}\ell_{i-1}\right|>Kd\sqrt{p_{i}\ell_{i-1}\ln p_{i}\ell_{i-1}},$$ 
		and $Y_{i,v}^{(b)}$ be the event 
		$$\left|d_{S[C_{i-1},C_{i}]}(v)-(m-1)p_{i}^{2}/2\right|>\frac{K}{2}\sqrt{(m-1)p_{i}^{2}/2\ln (m-1)p_{i}^{2}/2}.$$ 
		Then $d_{S[C_{i-1},C_{i}]}(v)\sim \text{Bi}(\frac{m-1}{2},p_i^2)$ and $\mathbb{E}(d_{S[C_{i-1},C_{i}]}(v))=(m-1)p_{i}^2/2\stackrel{(\ref{eq:p_i}),(\ref{ineq:elli})}{=}dp_{i}\ell_{i-1}\pm 1,$
		therefore $X_{i,v}^{b}\subseteq Y_{i,v}^{b}$ for large enough $m$. Moreover, Theorem~\ref{thm:Chernoff} implies that
		$$\mathbb{P}[X_{i,v}^{(b)}]\leq \mathbb{P}[Y_{i,v}^{(b)}]\leq  2((m-1)p_{i}^2/2)^{-K^2/12}\stackrel{(\ref{eq:p_i})}{=}O(m^{-2}).$$
		Finally, union bound yields $$\mathbb{P}[X^{(b)}]\leq \sum_{i\in[t],v\in C_i}\mathbb{P}[X_{i,v}^{(b)}]=o(1).$$
				
		{\bf Proof of Property (c).}
		Proof follows the lines of the proof of part (b). Since for $i\in [2,t]$ and $v\in C_{i}$ we have $d_{S[C_{i-1},C_i]}(v)\sim \text{Bi}(\frac{m-1}{2},2p_ip_{i-1})$ and $\mathbb{E}(d_{S[C_{i-1},C_i]}(v))=(m-1)p_{i}p_{i-1}=p_i\ell_{i-1}\pm 1$. Hence we have $\mathbb{P}[X^{(c)}]=o(1)$
		
		{\bf Proof of Property (d).}
		Proof follows the lines of the proof of part (b), since for all $v\in V(S)$ we have $d_{S[v\cup R]}(v)\sim \text {Bi}(\frac{m-1}{2}, \gamma^2)$ and $\mathbb{E}(d_{S[v\cup R]}(v))=\frac{m-1}{2} \gamma^2\stackrel{(\ref{ineq:gamma})}{\geq}2\rho(m-1)$. Hence we have $\mathbb{P}[X^{(d)}]=o(1)$
		
		{\bf Proof of Property (e)}

		We say that a set $C\subseteq V(S)$ is \emph{typical} if $|C|=p_0m\pm K\sqrt{p_0m\ln p_om}$ and $S[C]$ contains a copy of $T_0$. For a partition $\mathcal{P}=\{C_0, \ldots, C_{t}, R\}$ set $C_i(\mathcal{P})=C_i$ for all $i\in[0,t]$. 
		
		Next we will show that the first statement of (e), namely that $C_0(\mathcal{P})$ is typical, holds asymptotically almost surely.		
		\begin{claim}\label{claim:e1}
			$$\mathbb{P}[C_0(\mathcal{P})\; \text{is typical }]=1-o(1).$$
		\end{claim}
		\begin{proof}
			Let $X$ be the event that $||C_0|-p_0m|\leq K\sqrt{p_0m\ln p_0m},$ 
			and $Y$ be the event that $S[C_0(\mathcal{P})]$ contains a copy of $T_0$. Since $|C_0|\sim \text{Bi}(m,p_0)$ and $\mathbb{E}(|C_0|)=p_0m\stackrel{(\ref{ineq:elli})}{=}\Omega(\sqrt{m})$, Theorem~\ref{thm:Chernoff} implies $\mathbb{P}[X]=1-o(1)$.
			
			For $v\in V(S)$ let $Z_v$ denote the event that $d_{S[C_0]}(v)\geq |V(T_0)|$, then $\bigcap_{v\in V(S)}Z_v\subseteq Y$. Indeed, if every vertex has degree at least $|V(T_0)|$ is $S[C_0]$, then $T_0$ can be found in $S[C_0]$ greedily, adding one hyperedge at a time.
			
			Following the lines of proof of (d), we have $d_{S[C_0]}(v)\sim \text{Bi}(\frac{m-1}{2},p_0^2)$, and $$\mathbb{E}(d_{S[C_0]}(v))=\frac{m-1}{2}p_0^2\stackrel{(\ref{ineq:p_0})}{\geq}2|V(T_0)|,$$
			so by Theorem~\ref{thm:Chernoff} for all $v\in V(S)$ we have $\mathbb{P}[Z_v]\geq 1-o(m^{-20}).$ Finally, 
			$$\mathbb{P}[Y]\geq \mathbb{P}[\bigcap_{v\in V(S)}Z_v]\geq 1-m\cdot o(m^{-20})\geq 1-o(1).$$ Therefore $\mathbb{P}[X]=\mathbb{P}[Y]=1-o(1)$ and hence $\mathbb{P}[X\cap Y]=\mathbb{P}[C_0(\mathcal{P})\; \text{is typical }]=1-o(1)$.
		\end{proof}
		
		Now, for every typical set $C$, we fix one copy of $T_0$ in $S[C]$. 
		
		We first show that there are not many vertices in $\overline{C}=V(S)\setminus C$ that have low degree in $S[L_0, \overline{C}]$. 
		\begin{claim}\label{claim:count}
			For any $\alpha>0$ and any typical set $C$ all but at most $|C|/\alpha$ vertices in $v\in \overline{C}$ satisfy 
			$$d_{S[L_0,\overline{C}]}(v)=(1\pm\alpha)\ell_0.$$
		\end{claim}
		\begin{proof}[Proof of Claim]
			For $v\in \overline{C}$ and $x\in L_0$ there is a unique $w\in V(S)$ such that $\{v,x,w\}\in E(S)$. Consequently, $d_{S[L_0,\overline{C}]}(v)\leq \ell_0$ holds for any $v\in \overline{C}$.
			
			Let $A=\{\{v,x,w\} : x\in L_0, v,w \in \overline{C}\}$. Since for every $x\in L_0$ there are at most $|C|$ edges $\{v,x,w\}$ with $v\in \overline{C}$ and $w\in C$
			\begin{equation}\label{eq:|A|1}
				|A|\geq \ell_0(|\overline{C}|-|C|).
			\end{equation}
			On the other hand let $b$ be the number of ``bad'' vertices $v\in \overline{C}$, i.e., vertices $v$ with $d_{S[L_0,\overline{C}]}(v)<(1-\alpha)\ell_0$. Then we have 
			\begin{equation}\label{eq:|A|2}
				|A|\leq b(1-\alpha)\ell_0+(|\overline{C}|-b)\ell_0.
			\end{equation}
			
			Comparing~(\ref{eq:|A|1}) and ~(\ref{eq:|A|2}) yields that $b\leq|C|/\alpha$.
			\end{proof}	
			
			Let $E$ be the event that property (e) holds. Next we will show that 
			\begin{equation}\label{eq:conditional}
				\mathbb{P}[E | C_0(\mathcal{P})=C]=1-o(1) \; \text{for every typical } C.
			\end{equation}
			This implies that $E$ holds with probability $1-o(1)$. 
			
			Indeed, by Claim~\ref{claim:e1}, $\mathbb{P}[C_0(\mathcal{P})\text{ is typical}]=\sum_{C\text{ is typical}}\mathbb{P}[C_0(\mathcal{P})=C]=(1-o(1))$ and so
			\begin{align*}
				\mathbb{P}[E]&\geq \sum_{C\text{ is typical}}\mathbb{P}(C_0(\mathcal{P})=C)\mathbb{P}[E | C_0(\mathcal{P})=C]\\
				&\stackrel{(\ref{eq:conditional})}{\geq}(1-o(1)) \sum_{C\text{ is typical}}\mathbb{P}(C_0(\mathcal{P})=C)\geq 1-o(1).			
			\end{align*}
			It remains to prove~(\ref{eq:conditional}).

			Denote by $(\Omega, \mathcal{F}, \mathbb{P})$ the space of all partitions of $V(S)$ with $\mathbb{P}[v\in C_i]=p_i$ for $i\in[0,t]$ and $\mathbb{P}[v\in R]=\gamma$, and for fixed $C$ let $(\Omega, \mathcal{F}, \mathbb{P}_C)$ to be the space of all partitions of $V(S)$ with probability function $\mathbb{P}_C(A)=\mathbb{P}(A | C_0(\mathcal{P})=C)$.
				
			With this notation we need to show that $\mathbb{P}_{C}(E)=1-o(1)$ for every typical $C$. 
			
			Recall that $\overline{C}=V(S)/C$ and for all $v\in \overline{C}$ let 
			$$\chi(v)=\begin{cases}
					1, \text{ if } v\in C_1\\
					0, \text{ otherwise}.
				\end{cases}$$ 
			Note that for all $v\in \overline{C}$
				\begin{align*}
					\mathbb{P}_C(\chi(v)=1)&=\mathbb{P}_C(v\in C_1)=\mathbb{P}(v \in C_1 | C_0(\mathcal{P})=C)=\frac{\mathbb{P}(v \in C_1 \wedge C_0(\mathcal{P})=C)}{\mathbb{P}(C_0(\mathcal{P})=C)}\\ &=\frac{p_1\cdot p_0^{|C|}(1-p_0)^{m-|C|-1}}{p_0^{|C|}(1-p_0)^{m-|C|}} =\frac{p_1}{1-p_0}=q.
				\end{align*}
			Then by (\ref{eq:p_0})
				\begin{equation}\label{eq:q}
					q=(1\pm \varepsilon^{0.3})p_1.
				\end{equation}
				Moreover, since for fixed $v\in V(S)$ the event $\{v\in C_1\}$ was in the ``initial'' space $(\Omega, \mathcal{F}, \mathbb{P})$ independent of the outcome of random experiment for the remaining vertices $w\in V(S)\setminus\{v\}$, we infer that random variables $\{\chi(v) : v\in \overline{C}\}$ are mutually independent.
				
				Therefore for the rest of the proof we assume that typical $C$ with $L_0\subset C$ is fixed and all events and random variables are considered in the space $(\Omega, \mathcal{F}, \mathbb{P}_C)$.
				
			For a typical $C$ define 
					\begin{equation}\label{eq:M_C}
						M=M(C)=\{v\in \overline{C}: d_{S[L_0,\overline{C}]}(v)=(1\pm\varepsilon^{0.2})\ell_0\}.
					\end{equation}
					Recall that since $C$ is typical we have 
					\begin{equation}\label{eq:|C|}
						|C|=(1+o(1))p_0m, \text{ and } |\overline{C}|=(1-o(1))(1-p_0)m.
					\end{equation}
					Then by Claim~\ref{claim:count} with $\alpha=\varepsilon^{0.2}$ 
					\begin{equation}\label{ineq:|M_C|}
						|M|\geq \overline{C}-\frac{|C|}{\varepsilon^{0.2}}\stackrel{(\ref{eq:|C|})}{=}|\overline{C}|-\frac{|\overline{C}|p_0}{\varepsilon^{0.2}(1-p_0)}(1-o(1))\stackrel{(\ref{eq:p_0})}{\geq}(1-\varepsilon^{0.2})|\overline{C}|. 
				\end{equation}
				Note that $M$ is independent of choice of $C_1$ and is fully determined by $C$ and $S$.
				
				Next we verify that certain events $E^{(1)}$, $E^{(2)}$, $E^{(3)}$ hold asymptotically almost surely and that $E^{(1)}\wedge E^{(2)}\wedge E^{(3)}\subseteq E$. Let event $E^{(1)}$ be defined as 
				$$E^{(1)}: |M\cap C_1|\geq (1-\varepsilon^{0.1})|C_1| .$$
				
				Then $|C_1|\sim \text{Bi}(|\overline{C}|,q)$ and $|M\cap C_1|\sim \text{Bi}(|M|,q)$, so $$\mathbb{E}(|C_1|)=|\overline{C}|q \text{ and } \mathbb{E}(|M\cap C_1|)\stackrel{(\ref{ineq:|M_C|})}{\geq} (1-\varepsilon^{0.2})|\overline{C}|q.$$ Hence Theorem~\ref{thm:Chernoff} implies that with probability $1-o(1)$ we have $|M\cap C_1|/|C_1|\geq 1-\varepsilon^{0.1}$ and so $\mathbb{P}_C[E^{(1)}]=1-o(1)$.

				Now for every $v\in \overline{C}$ let $N(v)$ be the random variable that equals to the number of hyperedges $\{v,x,w\}$, where $x\in L_0$ and $w\in C_1$. Then $N(v)\sim \text{Bi}(d_{S[L_0,\overline{C}]}(v), q)$ for all $v\in \overline{C}$.
				
				Let $E^{(2)}$ be the event $$E^{(2)}: N(v)=(1\pm \varepsilon^{0.1})\ell_0p_1 \text{ for all } v\in M.$$ 
				
				For every $v\in M$, we have 
				$N(v)\sim \text{Bi}(d_{S[L_0,\overline{C}]}(v), q)$ and so $\mathbb{E}(N(v))=(1\pm2\varepsilon^{0.2})\ell_0p_1$ by (\ref{eq:M_C}) and (\ref{eq:q}). Then Theorem~\ref{thm:Chernoff} combined with union bound implies $\mathbb{P}_C[E^{(2)}]=1-o(1)$.
		
				Let $E^{(3)}$ be the event $$E^{(3)}: N(v)\leq 2\ell_0p_1 \text{ for all }v\in \overline{C}.$$ For every $v\in \overline{C}$ we have $N(v)\sim \text{Bi}(d_{S[L_0,\overline{C}]}(v), q)$ and $d_{S[L_0,\overline{C}]}\leq \ell_0$, hence we always have $\mathbb{E}(N(v))\stackrel{(\ref{eq:q})}{\leq} (1+\varepsilon^{0.3})\ell_0p_1.$
				Therefore by Theorem~\ref{thm:Chernoff} and union bound we have $\mathbb{P}_C[E^{(3)}]= 1-o(1)$.	 
				
				 It remains to notice that for $v\in C_1$ we have $d_{S[L_0,C_1]}(v)=N(v)$ and so $E^{(1)}\wedge E^{(2)}\wedge E^{(3)}\subseteq E$. Therefore, $\mathbb{P}_{C}[E]\geq 1-o(1)$, finishing the proof of (\ref{eq:conditional}). 	
	\end{proof}
	{\bf Embedding of $T$.} We start with applying Lemma~\ref{lemma:partition} to $S$ obtaining a partition $\mathcal{P}=\{C_0,\ldots, C_t, R\}$ of $V(S)$ that satisfy properties (a)-(e) of Lemma~\ref{lemma:partition}. To simplify our notation we set $G_1=S[L_0,C_1]$ and for $i\in[2,t]$ $G_i=S[C_{i-1},C_{i}]$.
	\begin{itemize}
		\item[1)] We first verify that Lemma~\ref{lemma:partition} guarantees that the assumptions of Lemma~\ref{lemma:FR} and Lemma~\ref{lemma:AKS} are satisfied. These Lemmas then yield systems of stars $\mathcal{S}_i=\{S_i^{1}, \ldots, S_{i}^{p_i}\}$ which cover almost all vertices of $G_i$. (See Figure~\ref{fig:1}, where each star $\mathcal{S}_i^{j}$ is a single grey edge.)
		\item[2)] Let $F$ be the union of $T_0$ with $\mathcal{S}_i$'s. The ``almost cover'' property of $\mathcal{S}_i$'s then allows us to show that hyperforest $F$ contains a large connected component $T_1$ which contains almost all vertices of $T$. (See Figure~\ref{fig:1}, green and grey edges form $T_1$.)
		\item[3)] Finally, we extend $T_1$ into a full copy of $T$ in a greedy procedure using the vertices of $R$. (See Figure~\ref{fig:1}, vertices of $R$ are blue.)
	\end{itemize}
		\begin{figure}[H]
		\begin{center}
			\begin{tikzpicture}[line cap=round,line join=round,>=triangle 45,x=1.0cm,y=1.0cm, scale=1]
			\clip(-7,-5) rectangle (7,5);
			\begin{scope}[xshift=0cm]
			\def\X{7cm};
			\def\AA{25};
			\def\AB{155};
			\def\AC{-10};
			\def\AD{-170};
			\def\AE{50};
			\def\AF{-80};
			\def\Y{5cm};
			\draw (0,0) ellipse ({\X} and {\Y});
			\draw [fill=black] (\AA:{\X} and {\Y}) circle (0pt);
			\draw [fill=black] (\AB:{\X} and {\Y}) circle (0pt);	
			
			\draw [fill=black] (\AC:{\X} and {\Y}) circle (0pt);	
			\draw [fill=black] (\AD:{\X} and {\Y}) circle (0pt);	
			\draw [fill=black] (\AE:{\X} and {\Y}) circle (0pt);
			\draw [fill=black] (\AF:{\X} and {\Y}) circle (0pt);
			\draw[color=black] (\AA:{\X} and {\Y})--(\AB:{\X} and {\Y});
			\draw (\AC:{\X} and {\Y}) .. controls (0,-1) and (0,-1) .. (\AD:{\X} and {\Y});
			\draw[color=black] (\AE:{\X} and {\Y})--(\AF:{\X} and {\Y});
			\filldraw[draw=white, fill=white] 
			(\AE:{\X} and {\Y})--(\AF:{\X} and {\Y})--(\AC:{\X} and {\Y})--(\AA:{\X} and {\Y})--cycle; 
			\draw[color=black] (\AE:{\X} and {\Y})--(\AF:{\X} and {\Y});
			\def \B{(0,0.9*\Y)}
			\def \BA{(100:{0.8*\X} and {0.8*\Y})}
			\def \BB{(80:{0.8*\X} and {0.8*\Y})}
			\def \BAA{(-0.3*\X, 0.6*\Y)}
			\def \BAB{(-0.2*\X, 0.6*\Y)}
			\def \BBA{(0.2*\X, 0.6*\Y)}
			\def \BBB{(0.3*\X, 0.6*\Y)}
			\filldraw[draw=green, fill=green, fill opacity=0.5, draw opacity=0.25] \B--\BA--\BB--cycle;
			\filldraw[draw=green, fill=green, fill opacity=0.5, draw opacity=0.25] \BA--\BAA--\BAB--cycle;
			\filldraw[draw=green, fill=green, fill opacity=0.5, draw opacity=0.25] \BB--\BBA--\BBB--cycle;
			\def \BAAA{(-0.8*\X, 0.1*\Y)}
			\def \BAAB{(-0.6*\X, 0.1*\Y)}
			\def \BABA{(-0.4*\X, 0.1*\Y)}
			\def \BABB{(-0.2*\X, 0.1*\Y)}
			\def \BBAA{(0.2*\X, 0.1*\Y)}
			\def \BBAB{(0.4*\X, 0.1*\Y)}
			\def \BBBA{(0.6*\X, 0.1*\Y)}
			\def \BBBB{(0.8*\X, 0.1*\Y)}
			\filldraw[draw=gray, fill=gray, fill opacity=0.5, draw opacity=0.25] \BAA--\BAAA--\BAAB--cycle;
			\filldraw[draw=gray, fill=gray, fill opacity=0.5, draw opacity=0.25] \BAB--\BABA--\BABB--cycle;
			\filldraw[draw=gray, fill=gray, fill opacity=0.5, draw opacity=0.25] \BBA--\BBAA--\BBAB--cycle;
			\filldraw[draw=blue, fill=blue, fill opacity=0.5, draw opacity=0.25] \BBB--\BBBA--\BBBB--cycle;
			\def \BAAAA{({cos(180)*0.9*\X},{-0.3*\Y})}
			\def \BAAAB{({cos(-160)*0.9*\X}, {sin(360-160)*0.3*\Y-0.3*\Y})}
			\def \BAABA{({cos(-150)*0.9*\X}, {sin(-150)*0.3*\Y-0.3*\Y})}
			\def \BAABB{({cos(-140)*0.9*\X}, {sin(-140)*0.3*\Y-0.3*\Y})}
			\def \BABAA{({cos(-130)*0.9*\X}, {sin(-130)*0.3*\Y-0.3*\Y})}
			\def \BABAB{({cos(-120)*0.9*\X}, {sin(-120)*0.3*\Y-0.3*\Y})}
			\def \BABBA{({cos(-110)*0.9*\X}, {sin(-110)*0.3*\Y-0.3*\Y})}
			\def \BABBB{({cos(-100)*0.9*\X}, {sin(-100)*0.3*\Y-0.3*\Y})}
			\def \BBAAA{({cos(-90)*0.9*\X}, {sin(-90)*0.3*\Y-0.3*\Y})}
			\def \BBAAB{({cos(-80)*0.9*\X}, {sin(-80)*0.3*\Y-0.3*\Y})}
			\def \BBABA{({cos(-70)*0.9*\X}, {sin(-70)*0.3*\Y-0.3*\Y})}
			\def \BBABB{({cos(-60)*0.9*\X}, {sin(-60)*0.3*\Y-0.3*\Y})}
			\def \BBBAA{({cos(-50)*0.9*\X}, {sin(-50)*0.3*\Y-0.3*\Y})}
			\def \BBBAB{({cos(-40)*0.9*\X}, {sin(-40)*0.3*\Y-0.3*\Y})}
			\def \BBBBA{({cos(-25)*0.9*\X}, {sin(-25)*0.3*\Y-0.3*\Y})}
			\def \BBBBB{({cos(0)*0.9*\X}, {sin(0)*0.3*\Y-0.3*\Y})}
			\filldraw[draw=gray, fill=gray, fill opacity=0.5, draw opacity=0.25] \BAAA--\BAAAA--\BAAAB--cycle;
			\filldraw[draw=gray, fill=gray, fill opacity=0.5, draw opacity=0.25] \BAAB--\BAABA--\BAABB--cycle;
			\filldraw[draw=gray, fill=gray, fill opacity=0.5, draw opacity=0.25] \BABA--\BABAA--\BABAB--cycle;
			\filldraw[draw=gray, fill=gray, fill opacity=0.5, draw opacity=0.25] \BABB--\BABBA--\BABBB--cycle;
			\filldraw[draw=gray, fill=gray, fill opacity=0.5, draw opacity=0.25] \BBAA--\BBAAA--\BBAAB--cycle;
			\filldraw[draw=blue, fill=blue, fill opacity=0.5, draw opacity=0.25] \BBAB--\BBABA--\BBABB--cycle;
			\filldraw[draw=blue, fill=blue, fill opacity=0.5, draw opacity=0.25] \BBBA--\BBBAA--\BBBAB--cycle;
			\filldraw[draw=blue, fill=blue, fill opacity=0.5, draw opacity=0.25] \BBBB--\BBBBA--\BBBBB--cycle;
			\draw [fill=black] \B circle (2pt);
			\draw [fill=black] \BA circle (2pt);
			\draw [fill=black] \BB circle (2pt);
			\draw [fill=black] \BAA circle (2pt);
			\draw [fill=black] \BAB circle (2pt);
			\draw [fill=black] \BBA circle (2pt);
			\draw [fill=black] \BBB circle (2pt);
			\draw [fill=black] \BAAA circle (2pt);
			\draw [fill=black] \BAAB circle (2pt);
			\draw [fill=black] \BABA circle (2pt);
			\draw [fill=black] \BABB circle (2pt);
			\draw [fill=black] \BBAA circle (2pt);
			\draw [fill=black] \BBAB circle (2pt);
			\draw [fill=blue] \BBBA circle (2pt);
			\draw [fill=blue] \BBBB circle (2pt);
			\draw [fill=black] \BAAAA circle (2pt);
			\draw [fill=black] \BAAAB circle (2pt);
			\draw [fill=black] \BAABA circle (2pt);
			\draw [fill=black] \BAABB circle (2pt);
			\draw [fill=black] \BABAA circle (2pt);
			\draw [fill=black] \BABAB circle (2pt);
			\draw [fill=black] \BABBA circle (2pt);
			\draw [fill=black] \BABBB circle (2pt);
			\draw [fill=black] \BBAAA circle (2pt);
			\draw [fill=black] \BBAAB circle (2pt);
			\draw [fill=blue] \BBABA circle (2pt);
			\draw [fill=blue] \BBABB circle (2pt);
			\draw [fill=blue] \BBBAA circle (2pt);
			\draw [fill=blue] \BBBAB circle (2pt);
			\draw [fill=blue] \BBBBA circle (2pt);
			\draw [fill=blue] \BBBBB circle (2pt);
			\draw (0,3) ellipse ({0.4*\X} and {0.1\Y});
			\draw [fill=black]  \BAA++(-1,0) circle (0pt) node {$L_0$};
			\draw [fill=black]  (0,0.47*\Y) circle (0pt) node {$C_0$};
			\draw [fill=black]  (0,0) circle (0pt) node {$C_1$};
			\draw [fill=black]  (0,-0.8*\Y) circle (0pt) node {$C_2$};
			\draw [fill=black]  (0.9*\X,0) circle (0pt) node {$R$};
			\draw [fill=black]  (-0.8*\X,0.8*\Y) circle (0pt) node {$S$};
			\end{scope}
			
			\end{tikzpicture}
		\end{center}
		\caption{Case $d=1$ and $t=2$. Green edges form $T_0$, grey edges are hyperstars $S_i^j$, blue edges are constructed by using vertices in reservoir $R$.}\label{fig:1}
	\end{figure}
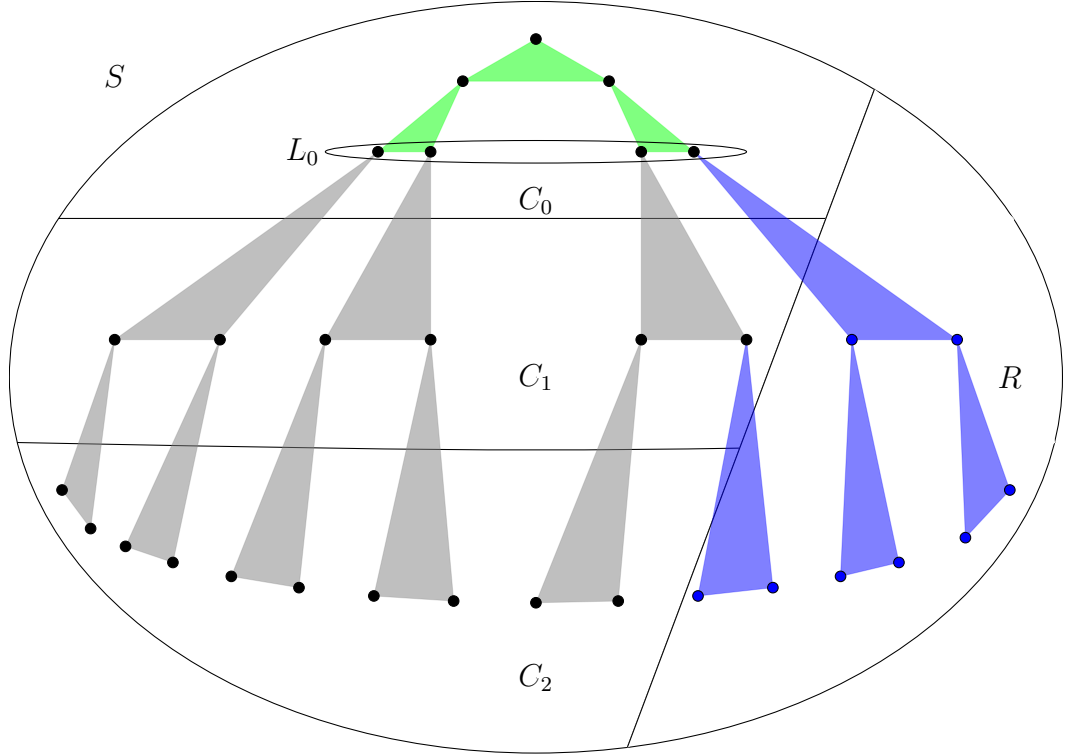	
	{\bf Step 1.} Construction of hyperforest $F$. 
	
	We start with applying Lemma~\ref{lemma:partition} to $S$ and obtaining a partition $\mathcal{P}=\{C_0,\ldots, C_t, R\}$ of $V(S)$ that satisfies properties (a)-(e) of Lemma~\ref{lemma:partition}. To simplify our notation, let $G_1=S[L_0,C_1]$ and for $i\in[2,t]$ let $G_i=S[C_{i-1},C_i]$. 
	
	Let $N_i=|V(G_i)|$, then $N_1=\ell_0+|C_1|$ and $N_{i}=|C_{i-1}|+|C_{i}|$ for $i\in[2,t]$. Due to property (a) of Lemma~\ref{lemma:partition} we have that for sufficiently large $m$ and for all $i\in[t]$
	\begin{equation}\label{eq:N_i}
		N_i=(1\pm \varepsilon)(\ell_{i-1}+\ell_{i})\leq(1\pm \varepsilon)\left(\frac{\ell_i}{2d}+\ell_{i}\right)\leq 2\ell_i.
	\end{equation} 
	In what follows we will show that $G_1$ satisfies assumptions of Lemma~\ref{lemma:FR} and for $i\in[2,t]$, $G_i$ satisfies assumptions of Lemma~\ref{lemma:AKS}. 
	
	We start with $G_1$. Recall that given $\mu$ we defined $\delta$ (see (\ref{eq:rho+delta})), and $\varepsilon_{3.1}$ was a constant guaranteed by Lemma~\ref{lemma:FR} with $\delta$ and $k=2$ as input. We then defined $\varepsilon$ (see (\ref{ineq:epsilon})) to be small enough, in particular such that $\varepsilon_{3.1}\geq\varepsilon^{0.1}$. We will now show that $G_1$ satisfies conditions of Lemma~\ref{lemma:FR} with $D=D_1=p_1\ell_0$ and $N=N_1$. 
	
	To verify condition (i) of Lemma~\ref{lemma:FR} it is enough to recall that $G_1=S[L_0,C_1]$ and so (i) holds with $X=L_0$ and $Y=C_1$.
	
	Note that condition (iii) is guaranteed by property (e), since for all $v\in C_1=Y$ we have $d_{G_1}(v)\leq 2D$.
	
	We now verify condition (ii) of Lemma~\ref{lemma:FR}. Property (b) with $i=1$ guarantees that for all $v\in X=L_0$
	$$d_{G_1}(v)=d\left(D\pm K\sqrt{D\ln D}\right).$$
	Since $D=p_1\ell_0\stackrel{(\ref{ineq:elli}),(\ref{eq:p_i})}{=}\Omega(\sqrt{n})$, for large enough $n$ we have for all $v\in X=L_0$
	\begin{equation}\label{eq:d_{G_1}(x)}
		d_{G_{1}}(v)=dD(1\pm \varepsilon_{3.1}).	
	\end{equation} 	
	Property (e) in turn guarantees that for all but at most $\varepsilon^{0.1}|C_1|\leq \varepsilon_{3.1}N_1$ vertices $v\in Y=C_1$ we have 
	\begin{equation}\label{eq:d_{G_1}(y)}
		d_{G_{1}}(v)=(1\pm\varepsilon^{0.1})p_1\ell_0=D(1\pm \varepsilon_{3.1}).
	\end{equation} 
	Now, (\ref{eq:d_{G_1}(x)}) and (\ref{eq:d_{G_1}(y)}) imply that $G_1$ satisfies condition (ii) of Lemma~\ref{lemma:FR}.
	
	Lemma~\ref{lemma:FR} produces a collection $\{S^{1}_1, \ldots, S^{1}_{p_1}\}$ of disjoint hyperstars of $G_1$ centered at vertices of $L_0$, each $S^1_j$ has at most $d$ hyperedges, and 
	\begin{equation}\label{eq:S_1}
		\text{star forest $\mathcal{S}_1=\bigcup_{j=1}^{p_1}S^{1}_j$ 
			covers all but at most $\delta N_1$ vertices of $G_1$.}
	\end{equation}
	
	Similarly for all $i\in[2,t]$, hypergraph $G_i$ satisfies assumptions of Lemma~\ref{lemma:AKS} with the parameters $K_{3.2}=8, d_{3.2}=d$ and $D=D_i=p_{i}\ell_{i-1}$. Indeed, condition (i) of Lemma~\ref{lemma:AKS} is guaranteed by taking $X=C_{i-1}$ and $Y=C_{i}$, and condition (ii) is guaranteed by properties (b) and (c) of Lemma~\ref{lemma:partition}.
	
	Then for $i\in[2,t]$, Lemma~\ref{lemma:AKS} when applied to $G_i$ yields a collection $\{S^{i}_1, \ldots, S^{i}_{p_i}\}$ of disjoint hyperstars centered at vertices of $C_{i-1}$, each $S^{i}_j$ has at most $d$ hyperedges and star forest $\mathcal{S}_i=\bigcup_{j=1}^{p_{i}}S_j$ 
	covers all but at most $O(N_iD_{i}^{-1/2}\ln^{3/2}D_{i})$ vertices of $G_i$.
	Once again recall that for $i\in[2,t]$ we have $D_i=p_i\ell_{i-1}\stackrel{(\ref{ineq:elli}),(\ref{eq:p_i})}{=}\Omega(n)$ and so for large enough $n$, 
	\begin{equation}\label{eq:S_i}
		\text{{$\mathcal{S}_{i}$ covers all but at most $\varepsilon N_i$ vertices of $G_i$.}}
	\end{equation}

	Now, let $F=T_0\cup\bigcup_{i=1}^{t}\mathcal{S}_{i}$, then $F$ is a hyperforest and we will find a hypertree $T_1\subseteq F$ such that $T_1$ is an almost spanning subhypertree of $T$.
	
	We now will estimate $|V(F)|$. Since for $i\in[2,t]$ a star forest $\mathcal{S}_{i}$ misses at most $\varepsilon N_{i}$ vertices of $C_i$ and $\mathcal{S}_{1}$ misses at most $\delta N_1$ vertices of $C_1$ we have:
	$$|V(F)|\geq |V(T_0)|+|C_1|-\delta N_1+\sum_{i=2}^{t}\left(|C_{i}|-\varepsilon N_i\right).$$
	
	Now, by property (a) of Lemma~\ref{lemma:partition}, for large enough $m$ and for any $i\in[t]$, we have $|C_i|=(1\pm\varepsilon)\ell_i$. Therefore
	$$	|V(F)|\geq |V(T_0)|+\sum_{i=1}^{t}(1-\varepsilon)\ell_i-\delta N_1 -\varepsilon\sum_{i=2}^{t}N_i
		  \stackrel{(\ref{eq:N_i})}{\geq} |V(T_0)|+\sum_{i=1}^{t}\ell_i-(\varepsilon+2\delta)\ell_1-3\varepsilon\sum_{i=2}^{t}\ell_{i}.$$
	Since $|V(T_0)|+\sum_{i=1}^{t}\ell_{i}= |V(T)|$ and $\varepsilon<\delta$ we have 
	\begin{equation}\label{ineq:|V(F)|}
		|V(F)|\geq |V(T)|-3\delta\sum_{i=1}^{t}\ell_{i}\geq(1-3\delta)|V(T)|.
	\end{equation}
	
	{\bf Step 2.} Embedding most of $T$.
	\begin{claim}\label{claim:sizeT1}
		$S$ contains a hypertree $T_1\subseteq F$ such that $T_1$ is a subhypertree of $T$ and $|E(T_1)|\geq (1-20\delta)|E(T)|$.
	\end{claim}
	\begin{proof}
		Recall that while $V(T)=\bigcup_{i=0}^{h}V_{i}$, the forest $F$ has the vertex set $V(F)=\bigcup_{i=0}^{i_0}V_{i}\cup \bigcup_{i=1}^{t}C_{i}$, where we have set $t=h-i_{0}$ and $V_{i_0}=C_0=L_0$. Also recall that $V_0=\{v_{0}\}$ was a root of $T$ (and $F$).
		
		For a non-root vertex $v\in V_{i}$ (or $v\in C_i$) a parent of $v$ is a unique vertex $u\in V_{i-1}$ (or $u\in C_{i-1}$) such that $\{u,v,w\}\in F$. If a vertex $v$ has a parent $u$ we will write $p(v)=u$, if a vertex $v$ has no parent we will say that $v$ is an orphan. 
		
		For each $v\in V(F)\setminus\{v_0\}$ consider ``a path of ancestors'' $v=a_{i}, a_{i-1},\ldots, a_1=a^{*}$, i.e., a path satisfying $p(a_j)=a_{j-1}$, $j\in [2,i]$ and such that $a^{*}=a^{*}(v)$ is an orphan in $F$. 
					
		Let $T_1\subseteq F$ be a subtree of $F$ induced on a set $\{v\in V(F): a^{*}(v)=v_{0}\}$. Note that if for some $v\in V(F)$ we have $a^{*}(v)\neq v_{0}$, then  $a^{*}(v)\not \in V(T_{0})$.			
					
		For $i\in [t]$ let $U_{i}\subseteq C_i$ be the set of vertices not covered by $\mathcal{S}_{i}$. Then
		\begin{equation}\label{eq:|U_i|}
			\text{$|U_1|\stackrel{(\ref{eq:S_1})}{\leq} \delta N_1$ and   
				$|U_{i}|\stackrel{(\ref{eq:S_i})}{\leq} \varepsilon N_{i}$ for all $i\in [2,t]$.}
		\end{equation}
		Note that all orphan vertices, except of $v_0$, belong to $\bigcup_{i=1}^{t} U_{i}$ as they were not covered by some $\mathcal{S}_i$. In particular, for every $v\not \in V(T_1)$ its ancestor $a^{*}(v)$ is an orphan and hence belongs to $\bigcup_{i=1}^{t} U_{i}$
		
		For an orphan vertex $a^{*}$ let $T(a^{*})$ be a subtree of $F$ rooted at $a^{*}$. Then for every orphan vertex $a^{*}\in U_{i}$ we have
		$$|V(T(a^{*}))|\leq 1+2d+\ldots+(2d)^{t-i}\leq 3(2d)^{t-i}.$$
		Finally, every $v\not \in V(T_1)$ is in $T(a^{*}(v))$, where $a^{*}(v)\in U_i$ for some $i\in[t]$, and therefore we have 
		$$ |V(T_1)| \geq |V(F)|-\sum_{i=1}^{t}|U_{i}|\cdot 3(2d)^{t-i}\stackrel{(\ref{ineq:|V(F)|}), (\ref{eq:|U_i|})}{\geq}(1-3\delta)|V(T)|-3\delta N_1 (2d)^{t-1}-3\varepsilon\sum_{i=2}^{t}N_{i}(2d)^{t-i}.$$
		Now, by (\ref{eq:N_i}), we have $N_{i}\leq 2 \ell_i$ and, by (\ref{ineq:elli}), $\ell_{i}(2d)^{t-i}=\ell_{t}$ for all $i\in[t]$, and so
	 	$$|V(T_1)|\geq (1-3\delta)|V(T)|-6\delta \ell_t -6\varepsilon(t-1)\ell_{t}.$$
	 	
		Recall that $\ell_{t}$ is the size of the last level of $T$, so $\ell_{t}\leq |V(T)|$. Also $t\leq 1+\log(\frac{1}{\varepsilon})$, and since $\varepsilon$ is sufficiently small $6\varepsilon(t-1)\leq \sqrt{\varepsilon}\leq \delta$ and so 
		$$|V(T_1)|\geq (1-10\delta)|V(T)|.$$
		Finally, for every hypertree $T^\prime$ we have $|V(T^\prime)|=2|E(T^\prime)|+1$, so we have
		\begin{equation*}
			|E(T_1)|\geq (1-20\delta)|E(T)|. \qedhere
		\end{equation*}
\end{proof}
{\bf Step 3.} Finally, we complete $T_1$ to a full copy of $T$ by using reservoir $R$. 
\begin{claim}\label{claim:reservoir}
	Assume that $\mathcal{P}$ is a partition guaranteed by Lemma~\ref{lemma:partition} and $T_1$ be a hypertree guaranteed by Claim~\ref{claim:sizeT1}. Then $T_1$ can be extended to a copy of $T$ in $S$.  
\end{claim}	
\begin{proof}
	Since $T_1$ is a subhypertree of $T$, there is a sequence of hyperedges $\{e_1,\ldots, e_{p-1}\}$ such that each $T_{i}=T_1\bigcup_{j=1}^{i-1}e_{j}$ for $i\in[p]$ is a hypertree an $T_p\cong T$. Every vertex $v\in V(S)$ has degree at least $\rho m$ in $R$ (by Lemma~\ref{lemma:partition}) and 
	$$p-1=|E(T)|-|E(T_1)|\stackrel{Claim~\ref{claim:sizeT1}}{\leq} 20\delta|E(T)|\leq 10\delta n \stackrel{(\ref{eq:rho+delta})}{\leq} \frac{\rho m}{2}.$$
	Hence we can greedily embed edges $e_1, \ldots, e_{p-1}$. Indeed, having embedded the edges $e_1,\ldots, e_{i-1}$ for some $i\in[p-1]$, for set $R_{i}=R\setminus \bigcup_{j=1}^{i-1}e_j$ and all $v\in V(S)$ we have
	$$d_{S[v\cup R_{i}]}\geq \rho m-2(i-1)>0$$
	allowing the greedy embeding to continue. Then the last hypertree $T_p$ is by construction isomorphic to $T$.
\end{proof}
	
\section{Concluding remarks}\label{sec:remarks}
We notice that with a similar proof one can verify Conjecture~\ref{conj:ER} for some other types of hypertrees.

Let $D=\{d_1, \ldots, d_{k}\}$ be a sequence of integers. Let $T$ be a tree rooted at $v_0$ and let $V(T)=V_0\sqcup V_1 \sqcup \ldots \sqcup V_{h}$ be a partition of $V(T)$ into levels, so that $V_i$ consist of vertices distance $i$ from $v_0$. We say that $T$ is \emph{$D$-ary hypertree} if for every $i\in[0,h-1]$ there is $j\in [k]$ such that for every $v\in V_i$ the forward degree of $v$ is $d_j$. In other words, forward degree of every non-leaf vertex of $T$ is in $D$, and depends only on the height of a vertex in $T$. 

Following the lines of the proof of Theorem~\ref{thm:main} one can conclude that for any finite set $D\subset \mathbb{N}$, and any $\mu$, any large enough STS $S$ contains any $D$-ary hypertree $T$, provided $|V(T)|\leq |V(S)|/(1+\mu)$.

\end{document}